\documentclass[12pt,reqno,oneside]{amsart}
\usepackage{amsmath, amssymb}
\usepackage{hyperref}
\newtheorem{theorem}{Theorem}[section]
\newtheorem{definition}{Definition}[section]
\newtheorem{example}{Example}[section]
\newtheorem{lemma}{Lemma}[section]

\newtheorem{corollary}{Corollary}[section]

\numberwithin{equation}{section}
\newcommand{\noi}{\noindent}

\begin{document}
\title[New Subclass of Pseudo-type Meromorphic Bi- Univalent functions...
  ]{New Subclass of Pseudo-type Meromorphic \\ Bi-Univalent functions of complex order\\ }\maketitle
\centerline{\bf {G. Murugusundaramoorthy$^{1,*}$ ,
    T. Janani,$^{2}$,and K. Vijaya$^3$   }}
\begin{center}$^{1,*}$  Corresponding Author\\School of Advanced Sciences,\\
VIT University, Vellore - 632 014,  India.\\
{\bf E-mail : gmsmoorthy@yahoo.com}\\
$^{2}$ School of Computer Science and Engineering\\
VIT University, Vellore - 632 014,  India.\\
{\bf E-mail: janani.t@vit.ac.in} \\
$^{3}$ School of Advanced Sciences,\\
VIT University, Vellore - 632 014,  India.\\
{\bf E-mail :kvijaya@vit.ac.in}\\
\end{center}
\begin{abstract}
In the present article, we define a new  subclass of pseudo-type meromorphic bi-univalent functions class $\Sigma'$ of complex order $\gamma \in \mathbb{C}\backslash \{0\}$ and  investigate the initial coefficient estimates $|b_0|, |b_1|$ and $|b_2|.$ Further we mention several new or known consequences of our result.\\
\\ \noi\textbf{AMS Subject Classification:} 30C45, 30C50.\\

\noi\textbf{Keywords:} Analytic functions; Univalent functions; Meromorphic functions;
Bi-univalent functions of complex order; Coefficient bounds, Pseudo functions.
\end{abstract}

\section{Introduction and Definitions}

Let $\mathcal{A}$ be class of analytic functions of the form
\begin{equation}\label{eq1}
f(z) = z + \sum_{n=2}^\infty a_nz^n
\end{equation}
which are univalent in the open unit disc $$\Delta = \{z :|z| < 1 \}.$$
Also, let $\mathcal{S}$ be class of all functions in $\mathcal{A}$ which are
univalent and normalized by the conditions $$f(0) = 0 = f'(0)-1$$in $\Delta.$ Some of the important and well-investigated subclasses of the
univalent function class $\mathcal{S}$ includes  the
class $\mathcal{S}^*(\alpha)(0\leq\alpha <1)$ of starlike functions of order
$\alpha$ in $\Delta$ and the class $\mathcal{K}(\alpha)(0\leq \alpha <1)$ of
convex functions of order $\alpha.$
\par An analytic function $\varphi$ is subordinate to an analytic function $\psi,$ written  by $\varphi(z) \prec \psi(z),$
provided there is an analytic function $\omega$ defined on $\Delta$ with $$\omega(0) = 0\qquad {\rm  and}\qquad
|\omega(z)|< 1 $$ satisfying $$\varphi(z) = \psi(\omega(z)).$$ Ma and Minda \cite{ma} unified various subclasses
of starlike and convex functions for which either of the quantity $$\frac{z\ f'(z)}{f(z)}\quad {\rm  or}\quad
1 +  \frac{z\ f''(z)}{f'(z)}$$ is subordinate to a more general superordinate function.
For this purpose, they considered an analytic function $\phi$ with positive real part in
the unit disk $\Delta, \phi(0) = 1, \phi'(0) > 0$ and $\phi$ maps $\Delta$
onto a region starlike with respect to 1 and symmetric with respect to the real axis.
\par The class of Ma-Minda starlike functions consists of functions $f \in \mathcal{A}$ satisfying
the subordination $$\frac{z\ f'(z)}{f(z)} \prec\phi(z).$$ Similarly, the class of Ma-Minda
convex functions consists of functions $f \in \mathcal{A}$ satisfying the subordination
$$1 +  \frac{z\ f''(z)}{f'(z)} \prec \phi(z).$$
 \par A function $f$ is bi-starlike of Ma-Minda
type or bi-convex of Ma-Minda type if both $f$ and $f^{-1}$ are respectively starlike
or convex. These classes are denoted respectively by $\mathcal{S}^{*}_{\Sigma}(\phi)$ and
$\mathcal{K}_{\Sigma}(\phi).$ In the sequel, it is assumed that $\phi$ is an analytic function with positive real part in
the unit disk $\Delta,$ satisfying $\phi(0) = 1, \phi'(0) > 0$ and $\phi(\Delta)$
is symmetric with respect to the real axis. Such a function has a series expansion of the form
\begin{equation}\label{c7e3}
\phi(z) = 1 + B_{1} z + B_{2} z^2 + B_{3} z^3 + \cdots,\ \ \ (B_{1}> 0).
\end{equation}
By setting $\phi(z)$ as given below:
\begin{equation}\label{phi01}
\phi(z) = \left( \frac{1 + z}{1 - z} \right)^{\alpha} = 1 + 2 \alpha z + 2 \alpha^2 z^2 +\frac{4\alpha^2+2\alpha}{3}z^3+ \cdots\quad
(0 < \alpha \leq 1),
\end{equation}
we have $$B_{1} = 2 \alpha , ~~~ B_{2} = 2 \alpha^2\quad {\rm and }\quad B_3=\frac{4\alpha^2+2\alpha}{3}.$$
\par On the other hand, if we take
\begin{equation}\label{phi02}\phi(z) =  \frac{1 + (1 - 2 \beta) z}{1 - z}  = 1 + 2 (1 - \beta) z + 2 (1 - \beta) z^2 + \cdots\quad(0 \leq  \beta < 1),
\end{equation}then $$B_{1} = B_{2} = B_3= 2(1 - \beta).$$
\par Let $\Sigma'$ denote the class of meromorphic univalent functions $g$ of the form
\begin{equation}\label{mero1}g(z) = z + b_0+ \sum_{n=1}^\infty \frac{b_n}{z^n}\end{equation}
defined on the domain $\Delta^*= \{z : 1 < |z| < \infty\}.$ Since $g \in \Sigma' $ is univalent, it has an
inverse $g^{-1}=h$ that satisfy $$g^{-1}(g(z)) = z, ~(z \in \Delta^*)$$ and
$$g(g^{-1}(w)) = w ,(M <|w| < \infty, M >0) $$ where
\begin{equation}\label{mero2}g^{-1}(w) = h(w) = w+\sum_{n=0}^\infty \frac{C_n}{w^n},\quad (M <|w| <\infty).\end{equation}
\par Analogous to the bi-univalent analytic functions, a function $g \in \Sigma'$ is said to be meromorphic
bi-univalent if $g^{-1} \in \Sigma'$. We denote the class of all meromorphic bi-univalent functions by $\mathcal{M}_{\Sigma'}.$
Estimates on the coefficients of meromorphic univalent
functions were widely investigated in the literature, for example,
Schiffer\cite{schiffer}  obtained the estimate $|b_2| \leq \frac{2}{3}$ for meromorphic univalent functions $g \in \Sigma'$ with $b_0 = 0$
and Duren \cite{duren} gave an elementary proof of the inequality $|b_n|\leq \frac{2}{(n+1)}$ on the coefficient of meromorphic
univalent functions $g \in\Sigma'$ with $b_k = 0$ for $1 \leq k < \frac{n}{2}.$
For the coefficient of the inverse of meromorphic univalent
functions $h \in \mathcal{M}_{\Sigma'}$, Springer \cite{springer} proved that
$|C_3| \leq 1 ;$  $|C_3+\frac{1}{2}C^2_1|\leq\frac{1}{2}$ and conjectured that $|C_{2n-1}|\leq\frac{(2n-1)!}{n!(n-1)!},\,\,(n = 1,2, ...).$
\par In 1977, Kubota \cite{kubota} has proved that the Springer's conjecture is true for $ n = 3,4,5 $ and subsequently Schober
\cite{schober}  obtained a sharp bounds for the coefficients $C_{2n-1}, 1\leq n \leq 7$ of the inverse of meromorphic univalent
functions in $\Delta^*.$ Recently, Kapoor and Mishra \cite{kapoor} (see \cite{kund}) found the coefficient estimates for a class consisting of inverses
of meromorphic starlike univalent functions of order $\alpha $ in $\Delta^*.$
\par Recently, Babalola \cite{Babalola} defined a new subclass $\lambda-$pseudo starlike function of order $\beta~(0\leq \beta<1)$ satisfying the analytic condition
\begin{equation}\label{Babalola}
\Re \left( \frac{z (f'(z))^{\lambda}}{f(z)} \right) > \beta, \qquad (z \in \mathbb{U},
\lambda \geq 1 \in \mathbb{R})
\end{equation}
and denoted by $\mathcal{L}_\lambda(\beta).$
Babalola \cite{Babalola} remarked that though for $\lambda > 1,$ these classes of $\lambda-$pseudo starlike
functions clone the analytic representation of starlike functions. Also, when $\lambda = 1,$ we have the class of starlike functions of order $\beta$($1-$pseudo starlike functions of order $\beta$) and for $\lambda= 2,$ we have the class of functions, which is a product combination of geometric expressions for bounded turning and starlike
functions.
\par Motivated by the earlier work of  \cite{DEN,jay1,kapoor,gmstjcho16,hms16,xu1}, in the present investigation, we define a new subclass of pseudo type meromorphic bi-univalent functions class $\Sigma'$ of complex order $\gamma \in \mathbb{C}\backslash \{0\}$, and the estimates
for the coefficients $|b_0|,|b_1|$ and $|b_2|$ are investigated.  Several new consequences of the results are also pointed out.\

\begin{definition}For $0 < \lambda \leq 1$ and $\mu \geq1,$
a function $g(z)\in\Sigma'$ given by (\ref{mero1}) is said to be in the class
$\mathcal{P}^\gamma_{\Sigma'}(\lambda,\mu,\phi)$ if the following conditions are
satisfied:
\begin{equation}\label{eq2}
1+\frac{1}{\gamma} \left[ (1-\lambda) \bigg(\frac{g(z)}{z}\bigg)^\mu + \lambda \left(  \frac{z (g'(z))^\mu}{g(z)} \right)
-1 \right] \prec \phi(z)
\end{equation}
and
\begin{equation}\label{eq3}
1+\frac{1}{\gamma} \left[ (1-\lambda) \bigg(\frac{h(w)}{w}\bigg)^\mu + \lambda \left(  \frac{w (h'(w))^\mu}{h(w)} \right)
-1 \right] \prec \phi(w)
\end{equation}
where $z,w \in \Delta^*, \, \gamma \in \mathbb{C}\backslash \{0\}$  and the function $h$ is given by (\ref{mero2}).
\end{definition}
\par By suitably specializing the parameter $\lambda,$ we state new subclass of meromorphic pseudo bi-univalent functions of complex order $\mathcal{P}^\gamma_{\Sigma'}(\lambda,\mu, \phi)$ as illustrated in the following Examples.
\begin{example}\label{exam1}For $\lambda=1,$ a function $g\in\Sigma'$ given by (\ref{mero1}) is said to be in the class  $\mathcal{P}^\gamma_{\Sigma'}(1,\mu,\phi)\equiv \mathcal{P}^\gamma_{\Sigma'}(\mu,\phi)$  if it satisfies the following conditions:
\begin{equation*}
1+\frac{1}{\gamma} \left( \frac{z (g'(z))^\mu}{g(z)}
-1 \right) \prec \phi(z)
\quad {\rm
and}\quad
1+\frac{1}{\gamma} \left(\frac{w (h'(w))^\mu}{h(w)}-1 \right) \prec \phi(w) \ \
\end{equation*}
where $z,w \in \Delta^*, \mu \geq 1, \, \gamma \in \mathbb{C}\backslash \{0\}$  and the function $h$ is given by (\ref{mero2}).
\end{example}
\begin{example}\label{exam2}For $\lambda=1$ and $\gamma=1,$ a function $g\in\Sigma'$ given by (\ref{mero1}) is said to be in the class  $\mathcal{P}^1_{\Sigma'}(1,\mu,\phi)\equiv \mathcal{P}_{\Sigma'}(\mu,\phi)$  if it satisfies the following conditions :
\begin{equation*}
\frac{z (g'(z))^\mu}{g(z)} \prec \phi(z)\quad {\rm
and}\quad
\frac{w (h'(w))^\mu}{h(w)}\prec \phi(w) \ \
\end{equation*}
where $z,w \in \Delta^*, \mu \geq 1$  and the function $h$ is given by (\ref{mero2}).
\end{example}

\section{Coefficient estimates for new function class   $\mathcal{P}^\gamma_{\Sigma'}(\lambda,\mu,\phi)$}
 \par In this section, we obtain the coefficient estimates $|b_0|,$ $|b_1|$ and $|b_2|$ for $\mathcal{P}^\gamma_{\Sigma'}( \lambda,\mu,\phi),$ a new  subclass of meromorphic pseudo bi-univalent functions class $\Sigma'$ of complex order $\gamma \in \mathbb{C}\backslash \{0\}.$
 In order to prove our result, we recall the  following lemma.

\begin{lemma}\label{lem1.1}\cite{Pom}
If $\Phi \in \mathcal{P}$, the class of all functions with $\Re \left(\Phi(z) \right)> 0 ,(z\in \Delta)$ then $$|c_k| \leq 2,\,\, {\rm for~~ each}\,\, k,$$
where
\[\Phi(z) = 1 + c_1z + c_2z^2 + \cdots \ \ \text{for} \ z \in \Delta.\]
\end{lemma}
\par Define the functions $p$ and $q$ in $\mathcal{P}$ given by
\begin{equation*}
p(z) = \frac{1 + u(z)}{1 - u(z)} = 1 + \frac{p_{1}}{z} + \frac{p_{2}}{z^2}+ \cdots
\end{equation*}
and
\begin{equation*}
q(z) = \frac{1 + v(z)}{1 - v(z)} = 1 + \frac{q_{1}}{z}  + \frac{q_{2}}{z^2}  + \cdots.
\end{equation*}
It follows that
\begin{equation*}
u(z) = \frac{p(z) - 1}{p(z) + 1} = \frac{1}{2}\left[ \frac{p_{1}}{z}+ \left(p_{2} - \frac{p_{1}^2}{2}\right) \frac{1}{z^2} + \cdots \right]
\end{equation*}
and
\begin{equation*}
v(z) = \frac{q(z) - 1}{q(z) + 1} = \frac{1}{2}\left[\frac{q_{1}}{z}  + \left(q_{2} - \frac{q_{1}^2}{2}\right) \frac{1}{z^2} +\cdots \right].
\end{equation*}
Note that for the functions $p(z),q(z) \in \mathcal{P},$  we have  \[|p_{i}| \leq 2 \,\,\,{\rm and}\,\,\, |q_{i}| \leq 2 \,\,\,\, {\rm for~~ each }\,\,i.\]

\begin{theorem}\label{th1}
Let $g$ be given by (\ref{mero1}) in the class $\mathcal{P}^\gamma_{\Sigma'}( \lambda,\mu,\phi).$
Then
\begin{equation}\label{eq7}
|b_0| \leq \frac{|\gamma ||B_1| }{|\mu-\mu\lambda-\lambda|},
\end{equation}
\begin{equation}\label{b1e}\small{
|b_1| \leq ~\frac{|\gamma|}{2|\mu-\lambda-2\mu\lambda|} \sqrt{4|(B_1-B_2)^2|+4|B^2_1|+8|B_1(B_1-B_2)|+\frac{|\mu(\mu-1)(1-\lambda)+2\lambda|^2|\gamma^2 B^4_1|}{|\mu-\mu\lambda-\lambda|^4}}}
\end{equation}
and
\begin{equation}\label{b2b}\small{
|b_2|\leq\frac{|\gamma|}{2|\mu-\lambda-3\mu\lambda|}\left(2|B_1|
+4|B_2-B_1|+2|B_1-2B_2+B_3|+\frac{|\mu(\mu-1)(\mu-2)(1-\lambda)-6\lambda||\gamma|^2|B_1|^3}{3|\lambda|^3}\right)}
\end{equation}
where $\gamma \in \mathbb{C}\backslash \{0\},0 < \lambda \leq 1, \mu \geq 1 $ and $\,\,z,w \in \Delta^*.$

\end{theorem}
\begin{proof}
It follows from (\ref{eq2}) and (\ref{eq3}) that
\begin{equation}\label{eq9}
1+\frac{1}{\gamma} \left[ (1-\lambda) \bigg(\frac{g(z)}{z}\bigg)^\mu + \lambda \left(  \frac{z (g'(z))^\mu}{g(z)} \right)
-1 \right]= \phi(u(z))
\end{equation}
and
\begin{equation}\label{eq10}
1+\frac{1}{\gamma} \left[ (1-\lambda) \bigg(\frac{h(w)}{w}\bigg)^\mu + \lambda \left(  \frac{w (h'(w))^\mu}{h(w)} \right)
-1 \right] = \phi(v(w)).
\end{equation}
In light of (\ref{mero1}), (\ref{mero2}), (\ref{eq2}) and (\ref{eq3}), we have
\begin{multline}\label{eq9d}
1+\frac{1}{\gamma} \left[ (1-\lambda) \bigg(\frac{g(z)}{z}\bigg)^\mu + \lambda \left(  \frac{z (g'(z))^\mu}{g(z)} \right)
-1 \right] \\=  1+B_{1} p_{1}\frac{1}{2z}+\left[\frac{1}{2} B_{1} \bigg(p_{2} - \frac{p_{1}^2}{2}\bigg) + \frac{1}{4} B_{2}p_{1}^2\right]\frac{1}{z^2}\qquad \qquad\qquad\qquad\qquad\\
+\left[\frac{B_1}{2}\left(p_3-p_1p_2+\frac{p_1^3}{4}\right)+\frac{B_2}{2}\left(p_1p_2-\frac{p_1^3}{2}\right)+B_3\frac{p_1^3}{8}\right]\frac{1}{z^3}
...\qquad{}\qquad{}
\end{multline} and
\begin{multline}\label{eq10d}
1+\frac{1}{\gamma} \left[ (1-\lambda) \bigg(\frac{h(w)}{w}\bigg)^\mu + \lambda \left(  \frac{w (h'(w))^\mu}{h(w)} \right)
-1 \right] \\=  1+B_{1} q_{1}\frac{1}{2w}+\left[\frac{1}{2} B_{1} \bigg(q_{2} - \frac{q_{1}^2}{2}\bigg) + \frac{1}{4} B_{2} q_{1}^2\right]\frac{1}{w^2}\qquad \qquad\qquad\qquad\qquad
\\+\left[\frac{B_1}{2}\left(q_3-q_1q_2+\frac{q_1^3}{4}\right)+\frac{B_2}{2}\left(q_1q_2-\frac{q_1^3}{2}\right)+B_3\frac{q_1^3}{8}\right]\frac{1}{w^3}...\qquad{}\qquad{}
\end{multline}
Now, equating the coefficients in (\ref{eq9d}) and (\ref{eq10d}), we get
\begin{equation}\label{eq13}
\frac{(\mu-\mu\lambda-\lambda)}{\gamma} b_0 = \frac{1}{2} B_{1} p_{1},
\end{equation}
\begin{equation}\label{eq14}
\frac{1}{2\gamma}\bigg[\big(\mu(\mu-1)(1-\lambda)+2\lambda\big)b_0^2+2(\mu-\lambda-2\lambda\mu) b_1\bigg]= \frac{1}{2} B_{1} \bigg(p_{2} - \frac{p_{1}^2}{2}\bigg) + \frac{1}{4} B_{2} p_{1}^2,
\end{equation}
\begin{multline}\label{eqb2}
\frac{1}{6\gamma}\bigg[\big(\mu(\mu-1)(\mu-2)(1-\lambda)-6\lambda\big) b_0^3+6\big(\mu(\mu-1)(1-\lambda)+2\lambda+\lambda\mu\big) b_0 b_1+6(\mu-\lambda-3\lambda\mu)b_2\bigg] \\ = \left[\frac{B_1}{2}\left(p_3-p_1p_2+\frac{p_1^3}{4}\right)+\frac{B_2}{2}\left(p_1p_2-\frac{p_1^3}{2}\right)+B_3\frac{p_1^3}{8}\right], \qquad \qquad \qquad
\end{multline}
\begin{equation}\label{eq15}
\frac{-(\mu-\mu\lambda-\lambda)}{\gamma} b_0 = \frac{1}{2} B_{1} q_{1},
\end{equation}
\begin{equation}\label{eq16}
\frac{1}{2\gamma}\bigg[\big(\mu(\mu-1)(1-\lambda)+2\lambda\big)b_0^2+2(\lambda-\mu+2\lambda\mu) b_1\bigg]= \frac{1}{2} B_{1} \bigg(q_{2} - \frac{q_{1}^2}{2}\bigg) + \frac{1}{4} B_{2} q_{1}^2
\end{equation}
and
\begin{multline}\label{eqb2a}
\frac{1}{6\gamma}\bigg[\big(6\lambda-\mu(\mu-1)(\mu-2)(1-\lambda)\big) b_0^3+6\big(\mu(\mu-1)(1-\lambda)-\mu(1-\lambda)+3\lambda+3\lambda\mu\big) b_0 b_1+6(\lambda-\mu+3\lambda\mu)b_2\bigg] \\ = \left[\frac{B_1}{2}\left(q_3-q_1q_2+\frac{q_1^3}{4}\right)+\frac{B_2}{2}\left(q_1q_2-\frac{q_1^3}{2}\right)+B_3\frac{q_1^3}{8}\right]\qquad \qquad \qquad
\end{multline}
From (\ref{eq13}) and (\ref{eq15}), we get
\begin{equation}\label{eq17}
p_1 = -q_1
\end{equation}
and
\begin{equation}\label{eqbo}
b_0^2 = \frac{\gamma^2 B_1^2}{8(\mu-\mu\lambda-\lambda)^2} (p_1^2 + q_1^2).\end{equation}
Applying Lemma (\ref{lem1.1}) for the coefficients $p_1$ and $q_1$, we
have
\[|b_0| \leq \frac{|\gamma ||B_1| }{|\mu-\mu\lambda-\lambda|}.\]

\par Next, in order to find the bound on $|b_1|$ from (\ref{eq14}), (\ref{eq16}) (\ref{eq17}) and(\ref{eqbo}), we obtain
\begin{multline}\label{eq19}
2(\mu-\lambda-2\lambda\mu)^2~\frac{b_1^2}{\gamma^2} + [\mu(\mu-1)(1-\lambda)+2\lambda]^2 \frac{b_0^4}{2\gamma^2}\\ = (B_1-B_2)^2~\frac{p^4_1}{8}+\frac{B_1^2}{4} (p^2_2+ q^2_2)+ B_1(B_2-B_1)\frac{(p_1^2p_2+q_1^2q_2)}{4}.\qquad{}\qquad{}
\end{multline}
Using (\ref{eqbo}) and applying Lemma (\ref{lem1.1}) once again for the coefficients $p_1, p_2$ and $q_2$, we get
\begin{multline*}|b_1|^2 \leq \frac{|\gamma^2|}{4|\mu-\lambda-2\lambda\mu|^2}\times\\
\left(4|(B_1-B_2)^2|+4|B_1|^2+8|B_1(B_1-B_2)|+\frac{|\mu(\mu-1)(1-\lambda)+2\lambda|^2|\gamma^2 B^4_1|}{|\mu-\mu\lambda-\lambda|^4}\right).\end{multline*}
That is,
\begin{multline*}|b_1| \leq ~\frac{|\gamma|}{2|\mu-\lambda-2\lambda\mu|} \times \\ \sqrt{4|(B_1-B_2)^2|+4|B_1|^2+8|B_1(B_1-B_2)|+\frac{|\mu(\mu-1)(1-\lambda)+2\lambda|^2|\gamma^2 B^4_1|}{|\mu-\mu\lambda-\lambda|^4}}.\end{multline*}

\par In order to find the estimate $|b_2|,$ consider the sum of (\ref{eqb2}) and (\ref{eqb2a}) with $p_1=-q_1,$ we have
\begin{equation}\label{b0b1}
\frac{1}{\gamma} b_0 b_1 =\frac{B_1[p_3+q_3]+(B_2-B_1)p_1[p_2-q_2]}{2[2\mu(\mu-1)(1-\lambda)-(1-\lambda)\mu+5\lambda+4\lambda\mu]}.
\end{equation}
Subtracting (\ref{eqb2a}) from (\ref{eqb2}) and using $p_1=-q_1$ we have
\begin{multline}\label{b2}
2(\mu-\lambda-3\lambda\mu)\frac{b_2}{\gamma}\\=-(\mu-\lambda-3\mu\lambda) \frac{b_0b_1}{\gamma} -[\mu(\mu-1)(\mu-2)(1-\lambda)-6\lambda] \frac{b_0^3}{3\gamma}+\frac{B_1}{2}(p_3-q_3)\\
+\frac{B_2-B_1}{2}(p_2+q_2)p_1+\frac{B_1-2B_2+B_3}{4}p_1^3.
\end{multline}
Substituting for $\frac{b_0b_1}{\gamma}$ and $\frac{b_0^3}{\gamma}$ in (\ref{b2}), simple computation yields,
\small{\begin{eqnarray}\label{b2a}
\frac{b_2}{\gamma}&=&\frac{-B_1}{2(\mu-\lambda-3\lambda\mu)}\left(\frac{\mu-3\lambda-4\lambda\mu-\mu(\mu-1)(1-\lambda)}{2\mu(\mu-1)(1-\lambda)-\mu+5\lambda+5\lambda\mu}p_3+\frac{2\lambda+\lambda\mu+\mu(\mu-1)(1-\lambda)}{2\mu(\mu-1)(1-\lambda)-\mu+5\lambda+5\lambda\mu}q_3\right)\notag\\
&-&\frac{(B_2-B_1)p_1}{2(\mu-\lambda-3\lambda\mu)}\left(\frac{\mu-3\lambda-4\lambda\mu-\mu(\mu-1)(1-\lambda)}{2\mu(\mu-1)(1-\lambda)-\mu+5\lambda+5\lambda\mu}p_2-\frac{2\lambda+\lambda\mu+\mu(\mu-1)(1-\lambda)}{2\mu(\mu-1)(1-\lambda)-\mu+5\lambda+5\lambda\mu}q_2\right)\notag\\&+&
\frac{B_1-2B_2+B_3}{8(\mu-\lambda-3\lambda\mu)}p_1^3-\frac{(\mu(\mu-1)(\mu-2)(1-\lambda)-6\lambda)\gamma^2~B_1^3}{48(\mu-\lambda-3\lambda\mu)\lambda^3}p_1^3.
\end{eqnarray}}Applying Lemma \ref{lem1.1} in the above equation yields,
\begin{multline}
|b_2|\leq \frac{|\gamma|}{2|\mu-\lambda-3\lambda\mu|} \times \\ \left(2|B_1|
+4|B_2-B_1|+2|B_1-2B_2+B_3|+\frac{|\mu(\mu-1)(\mu-2)(1-\lambda)-6\lambda||\gamma|^2|B_1|^3}{3|\lambda|^3}\right)
\end{multline}
\end{proof}
\par By taking $\lambda =1,$ we  state the following.
\begin{theorem}\label{th6}
Let $g$ be given by (\ref{mero1}) in the class  $\mathcal{P}^\gamma_{\Sigma'}(\mu,\phi).$
Then
\begin{equation}\label{eq7e}
|b_0|\leq|\gamma|~|B_1|,
\end{equation}

\begin{equation}\label{eq8e}
|b_1| \leq ~\frac{|\gamma|}{|1+\mu|} \sqrt{|(B_1-B_2)^2|+|B^2_1|+2|B_1(B_1-B_2)|+|\gamma|^2~|B^4_1|}.
\end{equation}and
\begin{equation}\label{eq8f}|b_2| \leq \frac{|\gamma|}{|1+2\mu|}\left(|B_1|
+2|B_2-B_1|+|B_1-2B_2+B_3|+|\gamma|^2~|B_1|^3\right)\end{equation}
where $\gamma \in \mathbb{C}\backslash \{0\}, \mu \geq 1$ and $\,\,z,w \in \Delta^*.$
\end{theorem}
\par By taking $\lambda =1$ and $\gamma=1,$ we  state the following results.
\begin{theorem}\label{th6a}
Let $g$ be given by (\ref{mero1}) in the class  $\mathcal{P}_{\Sigma'}(\mu, \phi).$
Then
\begin{equation*}\label{eq7e}
|b_0|\leq |B_1|,
\end{equation*}
\begin{equation*}\label{eq8e}
|b_1| \leq \frac{1}{|1+\mu|}~\sqrt{|(B_1-B_2)^2|+|B^2_1|+2|B_1(B_1-B_2)|+|B^4_1|}.
\end{equation*}and
\begin{equation*}\label{eq8f}|b_2|\leq \frac{1}{|1+2\mu|}\left(|B_1|
+2|B_2-B_1|+|B_1-2B_2+B_3|+~|B_1|^3\right)\end{equation*}
where $\mu \geq 1, \,\,z,w \in \Delta^*.$
\end{theorem}

\section{\bf Corollaries and concluding Remarks}
 
\begin{corollary}\label{th01}
Let $g$ be given by (\ref{mero1}) in the class  $\mathcal{P}^\gamma_{\Sigma'}( \lambda,\mu,\left( \frac{1 + z}{1 - z} \right)^{\alpha})\equiv\mathcal{P}^\gamma_{\Sigma'}( \lambda,\mu,\alpha).$
Then
\begin{equation}
|b_0| \leq \frac{2|\gamma| \alpha }{|\mu-\mu\lambda-\lambda|},
\end{equation}

\begin{equation}
|b_1| \leq ~\frac{2|\gamma|\alpha}{|\mu-\lambda-2\lambda\mu|} \sqrt{(\alpha-2)^2+\frac{|\mu(\mu-1)(1-\lambda)+2\lambda|^2|\gamma^2|}{|\mu-\mu\lambda-\lambda|^4}\alpha^2}
\end{equation}and
\begin{equation}
|b_2|\leq \frac{2|\gamma|\alpha}{|\mu-\lambda-3\lambda\mu|}\left(3-2\alpha+\left(\frac{4-6\alpha+2\alpha^2}{3}\right)
+\frac{2|\gamma|^2\alpha^2|\mu(\mu-1)(\mu-2)(1-\lambda)-6\lambda|}{3|\lambda|^3}\right)
\end{equation}
where $\gamma \in \mathbb{C}\backslash \{0\},0 < \lambda \leq 1, \mu \geq 1$ and $\,\,z,w \in \Delta^*.$
\end{corollary}


\begin{corollary}\label{th31}
Let $g$ be given by (\ref{mero1}) in the class  $\mathcal{P}^\gamma_{\Sigma'}\left(\lambda,\mu,\frac{1 + (1 - 2 \beta) z}{1 - z}\right)\equiv\mathcal{P}^\gamma_{\Sigma'}( \lambda,\mu,\beta).$
Then
\begin{equation}
|b_0| \leq \frac{2|\gamma|(1-\beta)}{|\mu-\mu\lambda-\lambda|},
\end{equation}

\begin{equation}
|b_1| \leq ~\frac{2|\gamma|(1-\beta)}{|\mu-\lambda-2\lambda\mu|} \sqrt{1+\frac{|\mu(\mu-1)(1-\lambda)+2\lambda|^2|\gamma^2|}{|\mu-\mu\lambda-\lambda|^4}(1-\beta)^2}
\end{equation}and
\begin{equation}
|b_2|\leq\frac{2|\gamma|(1-\beta)}{|\mu-\lambda-3\lambda\mu|}\left(1+\frac{2|\gamma|^2(1-\beta)^2|\mu(\mu-1)(\mu-2)(1-\lambda)-6\lambda|}{3|\lambda|^3}
\right)
\end{equation}
where $\gamma \in \mathbb{C}\backslash \{0\},0 < \lambda \leq 1, \mu \geq 1$ and $\,\,z,w \in \Delta^*.$
\end{corollary}
{\bf Concluding Remarks:} We remark that, when $\lambda =1 $ and $\mu=1,$ the class $\mathcal{P}^\gamma_{\Sigma'}(\lambda,\mu,\phi)\equiv \mathcal{S}^\gamma_{\Sigma'}(\phi),$ the subclass of meromorphic bi-starlike functions of complex order. By taking $\mu=1$ in the Theorem \ref{th6}, we can easily obtain the coefficient estimates $b_0, b_1$ and $b_2$ for $\mathcal{S}^\gamma_{\Sigma'}(\phi),$ which leads to the results discussed in Theorem 2.3 of \cite{gmstjcho16}.
Also, we can obtain the initial coefficient estimates for function $g$ given by (\ref{mero1})  in the subclass $\mathcal{S}^\gamma_{\Sigma'}(\phi)$ by taking $ \phi(z)$ given in \eqref{phi01} and \eqref{phi02} respectively. 

\par {\bf Future Work:}
Let a function $g\in\Sigma'$ be given by (\ref{mero1}). By taking $\gamma=(1-\alpha) cos\beta~ e^{-i\beta},\,\,
|\beta|<\frac{\pi}{2},\,\, 0\leq\alpha < 1,$ the class $\mathcal{P}^\gamma_{\Sigma'}(\lambda,\mu,\phi)\equiv\mathcal{P}^\beta_{\Sigma'}(\alpha,\lambda,\mu,\phi)$ called the generalized class of $\beta$ bi-spiral like functions of order $\alpha (0\leq\alpha < 1)$ satisfying the following conditions.
\begin{equation*}\label{merospirala}
e^{i \beta}\left[(1-\lambda) \bigg(\frac{g(z)}{z}\bigg)^\mu + \lambda \left(  \frac{z (g'(z))^\mu}{g(z)} \right)
-1
 \right]\prec [\phi(z)(1-\alpha)+\alpha]cos~\beta +i sin \beta
\end{equation*}
and
\begin{equation*}\label{merospiral}
e^{i \beta}\left[  (1-\lambda) \bigg(\frac{h(w)}{w}\bigg)^\mu + \lambda \left(  \frac{w (h'(w))^\mu}{h(w)} \right)
-1 \right] \prec [\phi(w)(1-\alpha)+\alpha]cos~\beta +i sin \beta
\end{equation*}
where $ 0 < \lambda \leq 1 , ~\mu \geq 1, \,\,z,w \in \Delta^*$ and the function $h$ is given by (\ref{mero2}).
\par For function $g\in \mathcal{P}^\beta_{\Sigma'}(\alpha,\lambda,\mu,\phi)$  given by (\ref{mero1}), by choosing $\phi(z)=\big(\frac{1+z}{1-z}\big),({\rm or}\quad \phi(z)=\frac{1+Az}{1+Bz},-1 \leq B < A \leq 1),$ we can obtain the estimates $|b_0|,$  $|b_1|$ and $|b_2|$
  by  routine procedure (as in Theorem \ref{th1}) and so we omit the details.


\begin{thebibliography}{99}
\bibitem{Babalola} K. O. Babalola, On $\lambda-$pseudo starlike functions,
{\it Journal of Classical Analysis,} {3(2)}(2013), 137 -- 147.

\bibitem{duren} P. L. Duren, Coefficients of meromorphic schlicht functions,
{\it Proceedings of the American Mathematical Society,} 28(1971), 169--172.

\bibitem{DEN}E. Deniz, Certain subclasses of bi-univalent functions satisfying subordinate conditions,
{\it Journal of Classical Analysis,} {2}(1) (2013), 49--60.
\bibitem{jay1}T. Janani and G. Murugusundaramoorthy, Coefficient estimates of meromorphic bi- starlike
functions of complex order, {\it International Journal of Analysis and Applications,} 4(1)(2014), 68--77.

\bibitem{kapoor} G. P. Kapoor and A. K. Mishra, Coefficient estimates for
inverses of starlike functions of positive order, {\it Journal of
Mathematical Analysis and Applications,} 329(2)( 2007), 922--934.

\bibitem{kubota}Y. Kubota, Coefficients of meromorphic univalent functions,
{\it Kodai Mathematical Seminar Reports,} 28(2-3) (1977), 253--261.

\bibitem{ma} W.C. Ma, D. Minda, A unified treatment of some special classes of functions,
{\it in: Proceedings of the Conference on Complex Analysis, Tianjin,} 1992, 157 -- 169,
Conf. Proc.Lecture Notes Anal. 1. Int. Press, Cambridge, MA, 1994.

\bibitem{gmstjcho16} G. Murugusundaramoorthy, T. Janani and NE. Cho, Coefficient estimates of Mocanu type meromorphic bi-univalent functions of complex order, {\it Proceedings of the Jangjeon Mathematical Society,} {19}(2016), {691--700}.

\bibitem{Pom}C. Pommerenke, {\it Univalent Functions}, Vandenhoeck \&\ Ruprecht, G\"ottingen, 1975.

\bibitem{schober} G. Schober, Coefficients of inverses of meromorphic univalent functions,
 {\it Proceedings of the American Mathematical Society ,} 67(1)(1977), 111--116.

\bibitem{schiffer}M. Schiffer, On an extremum problem of conformal representation,
{\it Bulletin de la Société Mathématique
de France,} 66 (1938), 48--55.

\bibitem{springer}G. Springer, The coefficient problem for schlicht mappings of the exterior of the unit circle, {\it Transactions of the American Mathematical Society,} 70 (1951),
421--450.

\bibitem{kund} H. M. Srivastava, A. K. Mishra, and S. N. Kund, Coefficient
estimates for the inverses of starlike functions represented by
symmetric gap series, {\it Panamerican Mathematical Journal,} 21(4)(2011), 105--123.

\bibitem{hms16}
H. M. Srivastava, Santosh B. Joshi, Sayali S. Joshi and Haridas Pawar, Coefficient Estimates for Certain Subclasses of Meromorphically Bi-Univalent Functions, {\it Palestine Journal of Mathematics,} 5 (Special Issue: 1) (2016), 250--258.

\bibitem{xu1}Q-H. Xu , Chun-Bo Lv , H.M. Srivastava Coefficient estimates for the inverses of a certain general class of spirallike functions, {\it Applied Mathematics and Computation,} 219 (2013), 7000--7011.
\end{thebibliography}
\end{document}